\documentclass[10pt, reqno]{amsart}
\usepackage{amsmath, amsthm, amscd, amsfonts, amssymb, graphicx, color}
\usepackage[bookmarksnumbered, colorlinks, plainpages]{hyperref}
\hypersetup{colorlinks=true,linkcolor=red, anchorcolor=green, citecolor=cyan, urlcolor=red, filecolor=magenta, pdftoolbar=true}
\usepackage{mathrsfs}

\newtheorem{theorem}{Theorem}[section]
\newtheorem{lemma}[theorem]{Lemma}

\newtheorem{corollary}[theorem]{Corollary}
\theoremstyle{definition}

\newtheorem{example}[theorem]{Example}

\theoremstyle{remark}
\newtheorem{remark}[theorem]{Remark}
\numberwithin{equation}{section}

\begin{document}
\setcounter{page}{1}

\title[Remark on Faltings theorem]{Remark on Faltings theorem}

\author[Nikolaev]
{Igor V. Nikolaev$^1$}

\address{$^{1}$ Department of Mathematics and Computer Science, St.~John's University, 8000 Utopia Parkway,  
New York,  NY 11439, United States.}
\email{\textcolor[rgb]{0.00,0.00,0.84}{igor.v.nikolaev@gmail.com}}


\subjclass[2010]{Primary 11G35, 14A22; Secondary 46L85.}

\keywords{rational points, Serre $C^*$-algebras.}


\begin{abstract}
We prove  Faltings Finiteness Theorem using Rieffel's classification of the noncommutative tori. 
 \end{abstract}

\maketitle

\section{Introduction}
Diophantine geometry studies rational solutions of the  polynomial equations in terms of 
geometry of  the corresponding solutions in complex numbers [Hindry \& Silverman 2000]  \cite[Introduction]{HS}. 
The idea that arithmetic depends on geometry is amazingly fruitful, 
e.g. the Grothendieck's concept of a scheme over the Dedekind domain  led to the proof of Weil's 
Conjectures.    Usually  the arithmetic schemes corresponding to the polynomial rings with rational  coefficients support 
morphisms defined over the algebraic closure of the underlying field.  
As a result, such morphisms do not  preserve   rational points  on the algebraic varieties.

It is known since the works of J.~-P.~Serre and P.~Gabriel,  that algebraic geometry can be recast
in terms of non-commutative rings $R$ \cite[Chapter 13]{N}.  Unlike commutative rings, the  $R$ are 
 usually simple rings.  In particular,  the spectrum of the ring  $R$ is a singleton
 and  one cannot localize $R$ or define a scheme.
However  the simplicity of  $R$ is beneficial  for the diophantine geometry
 and the aim of our note is to show that the  morphisms of $R$ preserve rational points on  the  algebraic variety.
Such a property allows to recover  subtle arithmetic invariants  in terms of the invariants of the ring $R$,
see Section 4.  Let us recall some definitions.

Denote by $V$  a complex projective variety and let $\mathscr{A}_V$ be the
Serre $C^*$-algebra, i.e.  the norm closure of a self-adjoint representation of the twisted 
 homogeneous coordinate ring of  $V$ by the bounded linear operators on a Hilbert space
 \cite[Section 5.3.1]{N}. 
 Recall that the $C^*$-algebras  $\mathscr{A}$ and $\mathscr{A}'$ are said to be (strongly)  Morita equivalent, 
 if $\mathscr{A}\otimes\mathbf{K}\cong \mathscr{A}'\otimes\mathbf{K}$,
 where $\mathbf{K}$ is the $C^*$-algebra of all compact operators on a Hilbert space and 
 $\cong$ is the $C^*$-algebra isomorphism [Blackadar 1986]  \cite[13.7.1(c)]{B}. 
The Morita equivalence can be viewed as a generalization 
 of an isomorphism meaning 
 that the $C^*$-algebras $\mathscr{A}$ and $\mathscr{A}'$ have an isomorphic  category 
 of the finitely generated projective modules over $\mathscr{A}$ and $\mathscr{A}'$, respectively. 
 It is easy to see, that if $\mathscr{A}\cong\mathscr{A}'$ are isomorphic $C^*$-algebras, then they 
 are Morita equivalent  but not {\it vice versa}. 
 It is known, that the Morita equivalences of  the $C^*$-algebra $\mathscr{A}_V$
 preserve the complex points on $V$,  i.e. the projective variety $V$ is $\mathbf{C}$-isomorphic 
 to a projective variety $V'$ if and only if the corresponding Serre $C^*$-algebras are Morita 
 equivalent \cite[Theorem 5.3.3]{N}.  
 
 Let $k$ be a number field and let $V(k)$ be a projective variety over the field $k$. 
 Consider a   $C^*$-algebra $M_n(\mathscr{A}_{V(k)})$
consisting of the $n\times n$ matrices with the entries in $\mathscr{A}_{V(k)}$,
see  [Blackadar 1986]  \cite[p. 16]{B} for the details. 
 It follows from the K-theory of $\mathscr{A}_{V(k)}$,  that the $C^*$-algebra $M_n(\mathscr{A}_{V(k)})$ is simple, 
 see  lemma \ref{lm3.1}. 
 Denote by $K$ an extension of the field $k$ obtained by adjoining 
 the roots of all  polynomials of degree $n$ over $k$.
 Our main results can be stated as follows. 
\begin{theorem}\label{thm1.1}
 Projective variety $V(k)$ is $K$-isomorphic 
 to a projective variety $V'(k)$  if and only if 
 \begin{equation}\label{eq1.1}
 M_n\left(\mathscr{A}_{V(k)}\right)\cong  M_n\left(\mathscr{A}_{V'(k)}\right).
 \end{equation}
\end{theorem}
\begin{remark}\label{rmk1.2}
Roughly speaking, theorem \ref{thm1.1} says that the isomorphisms 
of the $C^*$-algebra  $M_n(\mathscr{A}_{V(k)})$ preserve the $K$-rational 
points on the variety $V(k)$. 
\end{remark}
\begin{corollary}\label{cor1.3}
Projective variety $V(k)$ is  $k$-isomorphic ($\mathbf{C}$-isomorphic, resp.) to a projective variety  $V'(k)$
if and only if the Serre $C^*$-algebra  $\mathscr{A}_{V(k)}$ is  isomorphic (Morita equivalent, resp.) to the Serre $C^*$-algebra
  $\mathscr{A}_{V'(k)}$. 
\end{corollary}
\begin{remark}\label{rmk1.4}
An independent proof of corollary \ref{cor1.3} using the Galois cohomology 
can be found in \cite{Nik1}. Our current proof is more elegant and follows from 
(\ref{eq1.1}) and an analog of the Zariski's Lemma for the algebra $\mathscr{A}_{V}$ (lemma \ref{lm3.2}). 
\end{remark}
\begin{corollary}\label{cor1.5}
{\bf (Faltings Theorem)}
If $\mathscr{C}(k)$ is a curve of genus $g\ge 2$, then the set of its $k$-points
is finite.  
\end{corollary}
The paper is organized as follows. The preliminary facts can be found in Section 2. 
Theorem \ref{thm1.1} and corollary \ref{cor1.3} are proved in Section 3. 
Section 4 contains proof of corollary \ref{cor1.5}.

\section{Preliminaries}
We briefly review noncommutative algebraic geometry, operator algebras
and Serre $C^*$-algebras. For a detailed exposition we refer the reader 
to [Stafford \& van ~den ~Bergh 2001]  \cite{StaVdb1}, [Blackadar 1986]  \cite[Chapter II]{B}
and \cite[Section 5.3.1]{N}, respectively. 

\subsection{Noncommutative algebraic geometry}
Let $V$ be a complex projective variety.
For an automorphism $\sigma: V\to V$ and an invertible
sheaf $\mathcal{L}$ of the linear forms on $V$,  one can construct  
a twisted homogeneous coordinate ring $B(V, \mathcal{L}, \sigma)$ of the variety $V$, 
i.e. a non-commutative ring such that:
\begin{equation}\label{eq2.1}
Mod~(B(V, \mathcal{L}, \sigma)) ~/~Tors \cong ~Coh~(V),  
\end{equation}
where $Mod$ is the category of graded left modules over the graded ring 
$B(V, \mathcal{L}, \sigma)$, $Tors$ the full subcategory of $Mod$ of the 
torsion modules and $Coh$ the category of quasi-coherent sheaves on the 
variety  $V$ [Stafford \& van ~den ~Bergh 2001]  \cite[p.180]{StaVdb1}.   
The correspondence $V\mapsto B(V, \mathcal{L}, \sigma)$ is a functor
which maps $\mathbf{C}$-isomorphic projective varieties to the Morita equivalent 
rings $B(V, \mathcal{L}, \sigma)$. 
\begin{example}\label{exm2.1}
Let $k$ be a field and $U_{\infty}(k)$ the algebra of polynomials
over $k$ generated by  two non-commuting variables $x_1$ and $x_2$ satisfying 
the  relation $x_1x_2-x_2x_1-x_1^2=0$.   Let $\mathbb{P}^1(k)$ be the projective
line over $k$.  Then $B(V, \mathcal{L}, \sigma)\cong U_{\infty}(k)$ and $V\cong \mathbb{P}^1(k)$
satisfy  (\ref{eq2.1}).   Notice  that the $B(V, \mathcal{L}, \sigma)$ is far from being a commutative ring.  
\end{example}
\begin{example}\label{exm2.2}
Let $S(\alpha,\beta,\gamma)$ be the  Sklyanin algebra, i.e   
a free $\mathbf{C}$-algebra on  four generators
$\{x_1, \dots, x_4\}$ satisfying six quadratic relations:
$x_1x_2-x_2x_1 = \alpha(x_3x_4+x_4x_3)$,
\quad $x_1x_2+x_2x_1 = x_3x_4-x_4x_3$,
~$x_1x_3-x_3x_1 = \beta(x_4x_2+x_2x_4)$,
~$x_1x_3+x_3x_1 = x_4x_2-x_2x_4$,
~$x_1x_4-x_4x_1 = \gamma(x_2x_3+x_3x_2)$ and 
$x_1x_4+x_4x_1 = x_2x_3-x_3x_2$,
where $\alpha,\beta,\gamma\in \mathbf{C}$ and  
$\alpha+\beta+\gamma+\alpha\beta\gamma=0$.  
Let  $\mathcal{E}\subset \mathbf{C}P^3$ 
be an elliptic curve given in the Jacobi form,  i.e.  as  an  intersection of two quartics:
\begin{equation}\label{eq2.2}
\left\{
\begin{array}{ccc}
u^2+v^2+w^2+z^2 &=&  0,\\
{1-\alpha\over 1+\beta}v^2+{1+\alpha\over 1-\gamma}w^2+z^2  &=&  0. 
\end{array}
\right.
\end{equation}
 Then $B(V, \mathcal{L}, \sigma)\cong S(\alpha,\beta,\gamma)$ and $V\cong \mathcal{E}$
satisfy  (\ref{eq2.1}). In other words,   the   $S(\alpha,\beta,\gamma)$ is a twisted homogenous 
coordinate ring of the elliptic curve (\ref{eq2.2}) modulo a two-sided ideal generated by the  central  elements
$\Omega_1 = x_1^2+x_2^2+x_3^2+x_4^2$ and 
$\Omega_2 = x_2^2+{1+\beta\over 1-\gamma}x_3^2+{1-\beta\over 1+\alpha}x_4^2$ 
 [Smith \& Stafford 1992]  \cite[p. 267]{SmiSta1}. 
\end{example}

\subsection{Operator algebras}
The $C^*$-algebra is an algebra  $\mathscr{A}$ over $\mathbf{C}$ with a norm 
$a\mapsto ||a||$ and an involution $\{a\mapsto a^* ~|~ a\in \mathscr{A}\}$  such that $\mathscr{A}$ is
complete with  respect to the norm, and such that $||ab||\le ||a||~||b||$ and $||a^*a||=||a||^2$ for every  $a,b\in \mathscr{A}$.  
Each commutative $C^*$-algebra is  isomorphic
to the algebra $C_0(X)$ of continuous complex-valued
functions on some locally compact Hausdorff space $X$. 
Any other  algebra $\mathscr{A}$ can be thought of as  a noncommutative  
topological space. 
\begin{example}\label{exm2.3}
An $n$-dimensional   noncommutative  torus  is the universal  $C^*$-algebra $\mathscr{A}_{\Theta_n}$
generated by $n$ unitary operators $U_1,\dots, U_n$ satisfying the commutation relations
$\{U_jU_i=\exp ~(2\pi i ~\theta_{ij}) ~U_iU_j ~|~1 \le i,j\le n\}$ which depend on a skew-symmetric matrix
\begin{equation}\label{skew}
\Theta_n=\left(
\begin{matrix}
              0 & \theta_{12}  & \dots & \theta_{1n}\cr
             -\theta_{12} & 0  & \dots & \theta_{2n}\cr
              \vdots         & \vdots         & \ddots   &\vdots\cr
             -\theta_{1n} & -\theta_{2n} & \dots & 0 
             \end{matrix}
              \right), 
              \quad\hbox{where}\quad  \theta_{ij}\in \mathbf{R}. 
              \end{equation}
\end{example}

\bigskip
By $M_{\infty}(\mathscr{A})$ 
one understands the algebraic direct limit of the $C^*$-algebras 
$M_n(\mathscr{A})$ under the embeddings $a\mapsto ~\mathbf{diag} (a,0)$. 
The direct limit $M_{\infty}(\mathscr{A})$  can be thought of as the $C^*$-algebra 
of infinite-dimensional matrices whose entries are all zero except for a finite number of the
non-zero entries taken from the $C^*$-algebra $\mathscr{A}$.
Two projections $p,q\in M_{\infty}(\mathscr{A})$ are equivalent, if there exists 
an element $v\in M_{\infty}(\mathscr{A})$,  such that $p=v^*v$ and $q=vv^*$. 
The equivalence class of projection $p$ is denoted by $[p]$.   
We write $V(\mathscr{A})$ to denote all equivalence classes of 
projections in the $C^*$-algebra $M_{\infty}(\mathscr{A})$, i.e.
$V(\mathscr{A}):=\{[p] ~:~ p=p^*=p^2\in M_{\infty}(\mathscr{A})\}$. 
The set $V(\mathscr{A})$ has the natural structure of an abelian 
semi-group with the addition operation defined by the formula 
$[p]+[q]:=\mathbf{diag}(p,q)=[p'\oplus q']$, where $p'\sim p, ~q'\sim q$ 
and $p'\perp q'$.  The identity of the semi-group $V(\mathscr{A})$ 
is given by $[0]$, where $0$ is the zero projection. 
By the $K_0$-group $K_0(\mathscr{A})$ of the unital $C^*$-algebra $\mathscr{A}$
one understands the Grothendieck group of the abelian semi-group
$V(\mathscr{A})$, i.e. a completion of $V(\mathscr{A})$ by the formal elements
$[p]-[q]$. 
\begin{example}\label{exm2.4}
{\bf ([Rieffel 1990] \cite{Rie1})}
\begin{equation}
K_0(\mathscr{A}_{\Theta_n})\cong \mathbf{Z}^{2^{n-1}}.
\end{equation} 
\end{example}

\subsection{Serre $C^*$-algebra}
Let $V$ be a complex projective variety and let  $B(V, \mathcal{L}, \sigma)$ be
 twisted homogeneous coordinate ring of $V$ (Section 2.1).  Denote by $\mathscr{H}$
 a Hilbert space and let $\mathbf{B}(\mathscr{H})$ be the algebra of bounded linear operators
 on $\mathscr{H}$.  Fix a self-adjoint (i.e. $\ast$-invariant) representation
 \begin{equation}\label{eq2.4}
 \rho: B(V, \mathcal{L}, \sigma)\longrightarrow \mathbf{B}(\mathscr{H}).
 \end{equation}
By the {\it Serre $C^*$-algebra}  of $V$ one understands the closure 
$\mathscr{A}_V$ of the $\ast$-algebra $\rho(B(V, \mathcal{L}, \sigma))$ 
in the norm topology on the space $\mathbf{B}(\mathscr{H})$. 
\begin{example}\label{exm2.5}
Let $V\cong \mathbf{A}_g$ be a $g$-dimensional abelian variety,
where $g\ge 1$.  The corresponding Serre $C^*$-algebra
$\mathscr{A}_V\cong \mathscr{A}_{\Theta_{2g}}$ is a $2g$-dimensional
noncommutative torus, see  Example \ref{exm2.3}. 
In particular for $g=1$ and $V\cong\mathcal{E}$ being an elliptic curve
as in Example \ref{exm2.2},  one gets  $\mathscr{A}_V\cong \mathscr{A}_{\theta}$
is an  irrational rotation algebra, i.e. the universal $C^*$-algebra on two generators $U_1$ and $U_2$
satisfying the relation $U_2U_1=e^{2\pi i\theta}U_1U_2$ for a constant $\theta\in\mathbf{R}$. 
This can be  proved by comparing the generators 
and relations of  the Sklyanin algebra $S(\alpha,\beta,\gamma)$ given in Example \ref{exm2.2}
with such for the algebra   $\mathscr{A}_{\theta}$ \cite[Section 1.3]{N}. 
\end{example}

\section{Proofs}
\subsection{Proof of theorem \ref{thm1.1}}
We split the proof in a series of  lemmas.

\begin{lemma}\label{lm3.1}
The  $M_n\left(\mathscr{A}_{V(k)}\right)$ is  
a simple $C^*$-algebra for each $n\ge 1$. 
\end{lemma}
\begin{proof}
(i) Let us prove  that the $\mathscr{A}_{V(k)}$ is a simple $C^*$-algebra.  
Indeed,  any Serre $C^*$-algebra $\mathscr{A}_{V}$ is a crossed product $C^*$-algebra \cite[Theorem 5.3.4]{N}. 
Denote by $\mathbb{A}$ an Approximately Finite-dimensional (AF-) $C^*$-algebra \cite[Section 3.5]{N},
such that   $\mathscr{A}_{V}\hookrightarrow \mathbb{A}$ is an embedding  and  $K_0(\mathscr{A}_{V})\cong K_0 (\mathbb{A})$. 
In particular, the $\mathscr{A}_{V(k)}$ embeds into an AF-algebra $\mathbb{A}$  of stationary type \cite[Section 3.5.2]{N}. 
Such AF-algebras are always simple,  since the corresponding dimension group $(K_0 (\mathbb{A}), K_0^+ (\mathbb{A}))$ 
\cite[Definition 3.5.2]{N} does not have order-ideals.  Since $K_0^+(\mathscr{A}_{V(k)})\cong K_0^+ (\mathbb{A})$,
we conclude that the Serre $C^*$-algebra  $\mathscr{A}_{V(k)}$ is  simple.

\medskip
(ii)  Let us prove  that the  $M_n\left(\mathscr{A}_{V(k)}\right)$ is a simple $C^*$-algebra.
Recall that there exists a one-to-one correspondence between the two-sided ideals of the 
matrix algebra $M_n\left(\mathscr{A}_{V(k)}\right)$ and the two-sided of  $\mathscr{A}_{V(k)}$. 
In particular, the  $M_n\left(\mathscr{A}_{V(k)}\right)$ is simple if and only if the  $\mathscr{A}_{V(k)}$
is simple. It follows from item (i), that the matrix $C^*$-algebra $M_n\left(\mathscr{A}_{V(k)}\right)$
is simple for each $n\ge 1$.  Lemma \ref{lm3.1} is proved. 
\end{proof}

\begin{lemma}\label{lm3.2}
{\bf (Zariski's Lemma for the algebra $\mathscr{A}_V$)}
 If  $K$ is an extension of the number field $k$ obtained  by adjoining 
 the roots of all  polynomials of degree $n$ over $k$,  then
 \begin{equation}
 \mathscr{A}_{V(K)}\cong M_n\left(\mathscr{A}_{V(k)}\right). 
 \end{equation}
\end{lemma}
\begin{proof}
(i) 
Recall that the $\mathscr{A}_{V(k)}$  is the norm-closure of
a self-adjoint representation of   the $k$-algebra 
$B:=B(V(k), \mathcal{L}, \sigma)$ of  non-commutative polynomials 
 over the number field $k$,  see   Example \ref{exm2.2} and formula (\ref{eq2.4}).  Consider 
 the Frobenius companion matrix given by an element 
\begin{equation}\label{eq3.2}
\mathscr{C}=
\left(
\begin{matrix}
0 & 0 & \dots & 0 &-c_0\cr
1 & 0 & \dots &  0 &-c_1\cr
\vdots &\vdots & & \vdots& \vdots\cr
0 & 0 & \dots & 1 & -c_{n-1}
\end{matrix}
\right), 
\qquad c_i\in k
\end{equation}
of the $K$-algebra $M_n (B)$. 
To determine the field of constants $K$ of the algebra  $M_n (B)$,
denote by  $I$ the identity of  $M_n (B)$
and consider the set $\{c I ~|~c\in K\}$.  It is easy to see, that 
\begin{equation}\label{eq3.3}
c I=\mathscr{C}\in   M_n (B)
\end{equation}
 if an only if 
$c$ is a root of the characteristic polynomial of matrix $\mathscr{C}$
given by (\ref{eq3.2}), i.e. 
\begin{equation}\label{eq3.4}
Char ~\mathscr{C}=x^n+c_{n-1}x^{n-1}+\dots+c_0.
\end{equation}

We conclude from (\ref{eq3.3})  that the field of constants $K$ of the algebra $M_n (B)$ 
must contain all roots of the  polynomial (\ref{eq3.4}). Since $c_i\in k$, these  roots are algebraic numbers of degree $n$ over $k$.

One can run through the set of all matrices $\mathscr{C}$ given by formula (\ref{eq3.2}) with $c_0\ne 0$. 
It is clear that in this way we get all algebraic extensions of degree $n$ over the number field $k$. 
Thus the field of constants $K$ of the algebra   $M_n (B)$
is obtained from $k$ by adjoining  all algebraic numbers of degree $n$ over $k$.

\bigskip
(ii) It follows from item (i) that the  $M_n (B)$ is an algebra of non-commutative polynomials
over the field $K$.  Since the field of constants of the twisted homogeneous coordinate ring
coincides with the field of definition of projective variety (see Example \ref{exm2.2}),
we conclude that 
\begin{equation}\label{eq3.5}
M_n(B)\cong B(V(K), ~\mathcal{L}, ~\sigma). 
\end{equation}

\bigskip
(iii) It remains to compare (\ref{eq3.5}) with the notation $B\cong B(V(k), ~\mathcal{L}, ~\sigma)$ 
and  to obtain a ring  isomorphism:
\begin{equation}\label{eq3.6}
B(V(K), ~\mathcal{L}, ~\sigma)\cong 
M_n(B(V(k), ~\mathcal{L}, ~\sigma)). 
\end{equation}
Taking the norm-closure of a self-adjoint representation (\ref{eq2.4}) 
 of the rings  in the both sides of (\ref{eq3.6}),  one gets a $C^*$-algebra 
 isomorphism: 
\begin{equation}\label{eq3.7}
\mathscr{A}_{V(K)}\cong M_n\left(\mathscr{A}_{V(k)}\right).
\end{equation}

Lemma \ref{lm3.2} is proved. 
\end{proof}

\begin{lemma}\label{lm3.3}
 The varieties $V(k)$ and $V'(k)$ are $K$-isomorphic
if and only if 
\begin{equation}
 M_n\left(\mathscr{A}_{V(k)}\right)\cong  M_n\left(\mathscr{A}_{V'(k)}\right).
\end{equation}
\end{lemma}
\begin{proof}
(i) Let $V(k)\cong V'(k)$ be a pair of isomorphic projective varieties defined over the
number field $k$. The canonical embedding 
\begin{equation}
\mathscr{A}_{V(k)}\hookrightarrow M_n\left(\mathscr{A}_{V(k)}\right)\cong \mathscr{A}_{V(K)}
\end{equation}
gives rise to an embedding $i$ of the variety $V(k)$ into the variety $V(K)$. 
Consider a commutative diagram in Figure 1, where $\psi$ is an isomorphism defined over 
the number field $K$. 
Since $\varphi=i\circ\psi\circ  i^{-1}$ is a composition of  maps  defined over the number field  $K$,  we conclude that the isomorphism 
$\varphi: V(k)\to V'(k)$ must also be  defined over the  field $K$.  
\begin{figure}
\begin{picture}(300,110)(-60,-5)
\put(30,70){\vector(0,-1){35}}
\put(140,70){\vector(0,-1){35}}
\put(55,23){\vector(1,0){60}}
\put(55,83){\vector(1,0){60}}
\put(15,20){$V'(k)$}
\put(128,20){$V'(K)$}
\put(17,80){$V(k)$}
\put(125,80){$V(K)$}
\put(80,30){$i$}
\put(80,90){$i$}
\put(10,50){$\varphi$}
\put(125,50){$\psi$}
\end{picture}
\caption{}
\end{figure}

\bigskip
(ii) Let $\psi$ be the $K$-isomorphism between the varieties $V(K)$ and $V'(K)$ shown on the diagram in 
Figure 1.  The $\psi $ gives rise to an isomorphism  $\mathscr{A}_{V(K)}\cong \mathscr{A}_{V'(K)}$
of the corresponding Serre $C^*$-algebras. 

In view of lemma \ref{lm3.2},  one gets the isomorphisms $\mathscr{A}_{V(K)}\cong M_n\left(\mathscr{A}_{V(k)}\right)$ and 
$\mathscr{A}_{V'(K)}\cong M_n\left(\mathscr{A}_{V'(k)}\right)$. 
Comparing the latter with the results of item (i), we conclude that 
the varieties $V(k)\cong V'(k)$ are $K$-isomorphic
if and only if
\begin{equation}
 M_n\left(\mathscr{A}_{V(k)}\right)\cong  M_n\left(\mathscr{A}_{V'(k)}\right).
\end{equation}

Lemma \ref{lm3.3} is proved. 
\end{proof}

\bigskip
Theorem \ref{thm1.1} follows from lemma \ref{lm3.3}.

\subsection{Proof of corollary \ref{cor1.3}}
(i) Let us prove that
the projective variety $V(k)$ is  $k$-isomorphic  to a projective variety  $V'(k)$
if and only if the Serre $C^*$-algebra  $\mathscr{A}_{V(k)}$ is  isomorphic  to the Serre $C^*$-algebra
  $\mathscr{A}_{V'(k)}$.  Indeed, suppose that in theorem \ref{thm1.1} we have  $n=1$. 
  In this case the field $K$ is an algebraic extension of the field $k$ of degree $1$.
  In other words, the fields $k\cong K$ are isomorphic. 
  
  On the other hand, theorem \ref{thm1.1} says that the variety $V(k)$ is $k$-isomorphic  to the variety 
  $V'(k)$ if and only if    $\mathscr{A}_{V(k)}\cong \mathscr{A}_{V'(k)}$ by formula 
  (\ref{eq1.1}) with $n=1$.  Item (i) is proved.

\bigskip
(ii) Let us prove that
the projective variety $V(k)$ is  $\mathbf{C}$-isomorphic to a projective variety  $V'(k)$
if and only if the Serre $C^*$-algebra  $\mathscr{A}_{V(k)}$ is  Morita equivalent  to the Serre $C^*$-algebra
  $\mathscr{A}_{V'(k)}$. 

We shall denote by $\bar k$ the algebraic closure of the number field $k$ obtained by adjoining the roots of
all polynomials of finite degree over $k$.  It is easy to see, that  this case of theorem \ref{thm1.1} corresponds to $n=\infty$ and 
$K\cong \bar k$.

Recall that the $C^*$-algebra $M_n\left(\mathscr{A}_{V(k)}\right)$ can be written in the form of a tensor product, i.e.
\begin{equation}\label{eq3.11}
M_n\left(\mathscr{A}_{V(k)}\right)\cong \mathscr{A}_{V(k)}\otimes M_n(\mathbf{C}). 
\end{equation}
  
 On the other hand, the canonical inclusions of the $C^*$-algebras $\mathbf{C}\subset M_2(\mathbf{C})\subset M_3(\mathbf{C})\subset\dots$
 are known to converge to an inductive limit: 
\begin{equation}
\lim_{n\to\infty} M_n(\mathbf{C})\cong \mathbf{K}, 
\end{equation}
where $\mathbf{K}$ is the $C^*$-algebra of all compact operators on a Hilbert space. 
Thus in view of (\ref{eq3.11}), the isomorphism (\ref{eq1.1}) for $n=\infty$ can be written in the form:
\begin{equation}\label{eq3.13}
\mathscr{A}_{V(k)}\otimes\mathbf{K} \cong 
\mathscr{A}_{V'(k)}\otimes\mathbf{K}.  
\end{equation}
By definition, the isomorphism (\ref{eq3.13}) says that the Serre $C^*$-algebras   $\mathscr{A}_{V(k)}$
and  $\mathscr{A}_{V'(k)}$  are Morita equivalent.

\medskip
It remains to apply the Lefschetz Principle to the number field $\bar k$.  Namely, one can pass from the algebraically
closed field $\bar k$  of characteristic zero to the field of complex numbers $\mathbf{C}$.   
This remark finishes the proof of item (ii).

\bigskip
(iii) Corollary \ref{cor1.3} follows from items (i) and (ii). 

\section{Faltings Theorem}
In this section we consider an application of the corollary \ref{cor1.3} 
to  an alternative proof of the Faltings Theorem about finiteness of the  
rational points  on the higher genus curves  [Hindry \& Silverman 2000]  \cite[Theorem E.0.1]{HS}. 
Namely, we show that corollary \ref{cor1.3} and a classification of the even-dimensional 
noncommutative tori up to an isomorphism (lemma \ref{lm4.2})  imply  the Shafarevich Conjecture for the abelian varieties [Par\v{s}in 1970] \cite[Conjecture S1]{Par1}. 
The Faltings Theorem is known to follow from the Shafarevich Conjecture using the  Parshin's Trick.  Let us pass to a detailed argument. 

\begin{theorem}\label{thm4.1}
{\bf (Shafarevich Conjecture)}
There exists only a finite number of the abelian varieties $\mathbf{A}_g$ over a number field $k$, such that:

\medskip
(i)   $\mathbf{A}_g$ have a  polarization of  degree $d\ge 1$;

\smallskip
(ii) $\mathbf{A}_g$ have a good reduction outside a fixed finite set $S$ of places of the field $k$. 
\end{theorem}
\begin{proof}
We split the proof in two lemmas having an independent interest.

\begin{lemma}\label{lm4.2} 
There exists only a finite number of pairwise distinct skew-symmetric matrices $\Theta_{2g}$ given by the formula (\ref{skew}),
such that the corresponding noncommutative tori $\mathscr{A}_{\Theta_{2g}}$ are isomorphic to each other. 
\end{lemma}
\begin{proof}
(i) Denote by $\rho_{ij}=\exp~(2\pi i\theta_{ij})$, see Example \ref{exm2.3}. 
Then the commutation relations for the noncommutative torus 
$\mathscr{A}_{\Theta_{2g}}$ can be written in the form:
\begin{equation}
U_iU_jU_i^{-1}U_j^{-1}=\rho_{ij}\in\mathbf{T},
\end{equation}
where $\mathbf{T}$ is the unit circle.  Since the $n$-dimensional 
noncommutative torus depends on ${n(n-1)\over 2}$ real parameters 
$\theta_{ij}$, one can identify the $\mathscr{A}_{\Theta_{2g}}$ with the
family of noncommutative tori
\begin{equation}\label{eq4.2}
\left\{\mathscr{A}_{\rho} ~|~ \rho\in\mathbf{T}^{g(2g-1)}\right\}. 
\end{equation}

\bigskip
(ii) 
 For a dense subalgebra $\mathscr{A}_{\rho}^0\subset \mathscr{A}_{\rho}$ consisting of the polynomials 
 in the variables $\{U_k ~|~ 1\le k\le 2g\}$ and an element $x\in \mathscr{A}_{\rho}^0$ 
 consider a function given by the formula
\begin{equation}\label{eq4.3}
\rho\mapsto ||x||. 
\end{equation}
It is easy to see,  that (\ref{eq4.3}) is a (real) analytic function on $\mathbf{T}^{g(2g-1)}$.

\bigskip
(iii)  Lemma \ref{lm4.2} can be proved by contradiction. 
Let  $\{\rho_k\}\in \mathbf{T}^{g(2g-1)}$ be an infinite sequence,  such that the $\mathscr{A}_{\rho_k}$ in 
 (\ref{eq4.2})  are pairwise isomorphic noncommutative tori. 
Since $\mathbf{T}^{g(2g-1)}$ is a compact space, one can always restrict to a Cauchy subsequence
$\{\rho_k'\}$ convergent to a point $\rho\in  \mathbf{T}^{g(2g-1)}$. 
Since function (\ref{eq4.3}) is constant on $\rho_k'$ and analytic on an interval $[\rho_k',\rho]$,
we conclude that (\ref{eq4.3}) is constant everywhere on the interval  $[\rho_k',\rho]$. 
This is obviously false,  since the interval $[\rho_k',\rho]\subset  \mathbf{T}^{g(2g-1)}$
contains  a continuum of pairwise non-isomorphic noncommutative tori $\mathscr{A}_{\rho}$
 [Rieffel 1990] \cite{Rie1}.  This contradiction finishes the proof of lemma \ref{lm4.2}. 
\end{proof}

\bigskip
\begin{example}\label{exm4.3}
Let $g=1$ in lemma \ref{lm4.2}. Then matrix (\ref{skew}) can be written as:
\begin{equation}
\Theta=\left(
\begin{matrix}
0 & \theta\cr
-\theta & 0
\end{matrix}
\right),  \qquad \hbox{where} \quad \theta\in\mathbf{R}. 
\end{equation}
It is known, that the noncommutative tori $\mathscr{A}_{\theta}$ and $\mathscr{A}_{\theta'}$ 
are isomorphic if and only if $\theta'=\pm\theta\mod\mathbf{Z}$ [Rieffel 1990] \cite[p. 200]{Rie1}.
In other words, there are only two distinct values $\rho_1,\rho_2\in\mathbf{T}$, such that 
$\rho_{1,2}=\exp (\pm 2\pi i\theta)$ correspond to the isomorphic noncommutative tori.   
\end{example}
\begin{remark}
For $g\ge 2$ an explicit formula for the matrices $\Theta_{2g}$ corresponding to the pairwise 
isomorphic noncommutative tori is an open problem  [Rieffel 1990] \cite[Question 4.1]{Rie1}.
Lemma \ref{lm4.2} can be viewed as a partial answer to this problem. 
\end{remark}

\begin{lemma}\label {lm4.5}
The number of abelian varieties  $\mathbf{A}_g$ over a number field $k$ satisfying conditions 
(i) and (ii) of Theorem \ref{thm4.1} is at most finite.  
\end{lemma}
\begin{proof}
(i) Let $\mathbf{A}_g$ be a $g$-dimensional abelian variety defined over the number field $k$.
Consider a $k$-isomorphism $\varphi:  \mathbf{A}_g\to \mathbf{A}_g'$ from  $\mathbf{A}_g$ 
to an abelian variety  $\mathbf{A}_g'$.  The corollary \ref{cor1.3} says that the Serre $C^*$-algebras
$\mathscr{A}_{\mathbf{A}_g}\cong\mathscr{A}_{\Theta_{2g}}$ and  $\mathscr{A}_{\mathbf{A}_g'}\cong\mathscr{A}_{\Theta_{2g}'}$
 must be isomorphic  noncommutative tori, see    Example \ref{exm2.5}.  The corresponding commutative 
diagram is shown in Figure 2. 

\begin{figure}
\begin{picture}(300,110)(-60,-5)
\put(30,70){\vector(0,-1){35}}
\put(140,70){\vector(0,-1){35}}
\put(55,23){\vector(1,0){60}}
\put(55,83){\vector(1,0){60}}
\put(20,20){$\mathscr{A}_{\Theta_{2g}}$}
\put(128,20){$\mathscr{A}_{\Theta_{2g}'}$}
\put(25,80){$\mathbf{A}_g$}
\put(130,80){$\mathbf{A}_g'$}
\put(80,30){$\cong$}
\put(80,90){$\varphi$}
\end{picture}
\caption{}
\end{figure}

\bigskip
(ii) 
Let $(\mathbf{A}_g, d, S)$ be the abelian variety $\mathbf{A}_g$ over the field $k$ with a polarization 
of degree $d\ge 1$ and good reduction outside a finite set $S$ of places of $k$.  
Since $\varphi$ is a $k$-isomorphism of $\mathbf{A}_g$, it must preserve the degree $d$ 
of polarization and the set $S$ of places of a bad reduction of $\mathbf{A}_g$. 
In other words,
\begin{equation}\label{eq4.5}
\varphi(\mathbf{A}_g, d, S)=(\mathbf{A}_g', d, S). 
\end{equation}

\bigskip
(iii) Lemma \ref{lm4.2} implies  that there exists at most  finite number of isomorphic  
abelian varieties  $\mathbf{A}_g$ satisfying (\ref{eq4.5}).  Indeed, for otherwise 
one gets an infinite number of isomorphic noncommutative tori  $\mathscr{A}_{\Theta_{2g}}$,
see Figure 2.  This contradicts lemma \ref{lm4.2} and finishes the proof of  
lemma \ref{lm4.5}. 
\end{proof}

\bigskip
Theorem \ref{thm4.1} follows from lemma \ref{lm4.5}. 
\end{proof}

\begin{corollary}\label{cor4.6}
{\bf (Faltings Theorem)}
If $\mathscr{C}(k)$ is a curve of genus $g\ge 2$, then the set of its $k$-points
is finite.  
\end{corollary}
\begin{proof}
The argument based on the Shafarevich Conjecture \ref{thm4.1} is well known.
For the sake of completeness, we repeat it in below. 

\bigskip
(i) One can always assume that there exists at least one $k$-point of $\mathscr{C}(k)$. 
The Jacobian $Jac~\mathscr{C}$ of $\mathscr{C}(k)$ is an  abelian 
variety over the field $k$.  Conversely,  given a principal polarization $d=1$ 
on  the Jacobian $Jac~\mathscr{C}$, one can recover a curve $\mathscr{C}(k)$ 
having the same set  $S$ as the abelian variety  $Jac~\mathscr{C}$. 
In view of Theorem \ref{thm4.1}, there exists only a finite number of 
pairs  $(\mathscr{C}(k), S)$ consisting of the isomorphic curves $\mathscr{C}(k)$
with a fixed set $S$. 

\bigskip
(ii) The rest of the proof is a Parshin's Trick described in  [Par\v{s}in 1970] \cite[Th\'eor\`eme 1]{Par1}. 
Namely, if $P\in\mathscr{C}(k)$ is a $k$-point,  then there exists a branched covering $\mathscr{C}_P(K)$
of  $\mathscr{C}(k)$, such that the number field $K$ is ramified in $S$ over $k$. 
The construction does not depend on the choice of the $k$-point $P$. 
It is known from item (i)  that there exists only a finite number of the pairs $(\mathscr{C}_P(K), S)$
and, therefore, only a finite number of the $k$-points on the curve  $\mathscr{C}(k)$. 

\bigskip
Corollary \ref{cor4.6} follows from item (ii). 
\end{proof}

\bigskip
\begin{remark}
It follows from the proof of theorem \ref{thm4.1} and corollary \ref{cor4.6}, 
that the number  $|\mathscr{C}(k)|$ of the rational points on the higher genus 
curve depends on  the number 
of isomorphic $n$-dimensional noncommutative tori  given  by the distinct skew-symmetric 
matrices (\ref{skew})  [Rieffel 1990] \cite[Question 4.1]{Rie1}.
To the best of our knowledge, the Rieffel's Problem is unsolved except for  the special cases  (see  example \ref{exm4.3}). 
\end{remark}

\bibliographystyle{amsplain}


\end{document}